\documentclass[a4paper,12pt]{article}
\usepackage[utf8]{inputenc}
\title{Continuous time integration for changing type systems}
\author{Sebastian Franz\footnote{
          Institute of Scientific Computing, Technische Universit\"at Dresden, Germany.
          \mbox{e-mail}: sebastian.franz@tu-dresden.de}
       }
\date{\today}

\usepackage{fancyhdr} 
\usepackage{amsmath}
\usepackage{amssymb}
\usepackage{amsthm}
\usepackage{cite}

\usepackage{booktabs}
\usepackage{pgfpages}

\newcommand{\e}{\mathrm{e}}
\newcommand{\I}{\mathcal{I}}
\newcommand{\N}{\mathbb{N}}
\newcommand{\R}{\mathbb{R}}
\newcommand{\U}{\mathcal{U}}
\newcommand{\V}{\mathcal{V}}
\newcommand{\PS}{\mathcal{P}}
\newcommand{\QS}{\mathcal{Q}}
\renewcommand{\H}{\mathbf{H}}

\newcommand{\dt}{\,\mathrm{d}t}
\newcommand{\vx}{\mathbf{x}}

\newcommand*{\dive}{\operatorname{div}}
\newcommand*{\grad}{\operatorname{grad}}
\newcommand*{\laplace}{\operatorname{\Delta}}
\newcommand{\scp}[1]{\left\langle #1 \right\rangle}
\newcommand{\scpr}[1]{\langle #1 \rangle_\rho}
\newcommand{\scprm}[1]{\langle #1 \rangle_{\rho,m}}
\newcommand{\id}[1]{\chi_{#1}}
\newcommand{\norm}[2]{\|{#1}\|_{#2}}
\newcommand{\tnorm}[1]{\left|\!\!\;\left|\!\!\;\left| {#1}
                       \right|\!\!\;\right|\!\!\;\right|}
\makeatletter
\renewcommand*\env@matrix[1][r]{\hskip -\arraycolsep
  \let\@ifnextchar\new@ifnextchar
  \array{*\c@MaxMatrixCols #1}}
\makeatother

\theoremstyle{plain}
\newtheorem{theorem}{Theorem}[section]
\newtheorem{lemma}[theorem]{Lemma}
\newtheorem{corollary}[theorem]{Corollary}
\newtheorem{remark}[theorem]{Remark}

\numberwithin{equation}{section}

\begin{document}
  \pagestyle{fancy}
  \maketitle
  \begin{abstract}
    We consider variational time integration using continuous Galerkin Petrov
    methods applied to evolutionary systems of changing type. We prove 
    optimal-order convergence of the error in a cGP-like norm and conclude the paper with 
    some numerical examples and conclusions.
  \end{abstract}


  
\section{Introduction}
  Let us start with an example, where the type of the problem changes over the spacial domain  
  and has homogeneous Dirichlet boundary conditions. For this purpose 
  let $n\in\{1,2,3\}$ be the spatial dimension and $\Omega\subset \R^n$ be bounded and partitioned 
  into measurable, disjoint sets $\Omega_{\mathrm{ell}},\Omega_{\mathrm{par}}$ and $\Omega_{\mathrm{hyp}}$. 
  In $\Omega_{\mathrm{hyp}}$ a hyperbolic wave equation is given for $U=(U_1,U_2)$
  \begin{align*}
    \partial_t U_1 + \dive(U_2) &= F_1,&
    \partial_t U_2 + \grad(U_1) &= F_2\text{ in }\Omega_{\mathrm{hyp}},
  \intertext{
  with some force term $F=(F_1,F_2)$. We will come to the boundary conditions for the spatial operators in a moment.
  In $\Omega_{\mathrm{par}}$ a parabolic heat equation is given
  }
    \partial_t U_1 + \dive(U_2) &= F_1,&
               U_2 + \grad(U_1) &= F_2\text{ in }\Omega_{\mathrm{par}},
  \intertext{
  and in $\Omega_{\mathrm{ell}}$ an elliptic reaction-diffusion equations completes the setting
  }
    U_1 + \dive(U_2) &= F_1,&
    U_2 + \grad(U_1) &= F_2\text{ in }\Omega_{\mathrm{ell}}.
  \end{align*}
  Each of above equations is also known in their derived second order formulation for $U_1$, namely
  $(\partial_t^2-\laplace)U_1=\partial_t F_1-\dive F_2$ for the wave equation, 
  $(\partial_t -\laplace)U_1=F_1-\dive F_2$ for the heat equation and
  $(1-\laplace)U_1=F_1-\dive F_2$ for the reaction-diffusion equation.\\
  Denoting by $\id{D}$ the characteristic function of a domain $D\subset\Omega$ 
  and defining the linear operators 
  \[
    M_0=\begin{pmatrix}[c]
            \id{\Omega_{\mathrm{hyp}}\cup\Omega_{\mathrm{par}}} & 0\\
            0 &\id{\Omega_{\mathrm{hyp}}}
        \end{pmatrix},\,
    M_1=\begin{pmatrix}[c]
          \id{\Omega_{\mathrm{ell}}} & 0 \\
          0 & \id{\Omega_{\mathrm{par}}\cup\Omega_{\mathrm{ell}}}
        \end{pmatrix}\quad\text{and}\quad
    A=\begin{pmatrix}[c]
          0 & \dive\\
          \grad^\circ &0
       \end{pmatrix},
  \]
  where $^\circ$ denotes the homogeneous Dirichlet boundary conditions w.r.t. $\Omega$,  
  we can write above equations in a condensed way
  \begin{subequations}\label{eq:evo}
  \begin{equation}
    (\partial_tM_0+M_1+A)U=F.
  \end{equation}
  By defining $A$ as above, we have included the boundary conditions at $\partial\Omega$ into $A$.
  All that is left is an initial condition at $t=0$ as we are only interested in $t\geq 0$:
  \begin{equation}
    M_0U(0^+)=M_0U_0.
  \end{equation}
  \end{subequations}
  Now we are left with the question, under which conditions above problem \eqref{eq:evo} has a unique solution.

  In the following we assume $U_0$ in $D(A)$.
  Besides that condition we can draw a condition on the operators from a much more general theory.  
  Most of the classical linear partial differential equations arising in mathematical physics
  can be written in a common operator form. It has been shown in \cite{Picard} that this form
  is an evolutionary problem, given by \eqref{eq:evo},
  where $\partial_t$ stands for the derivative with respect to time, 
  $M_0:\H\to\H$ and $M_1:\H\to\H$ are bounded linear selfadjoint operators on some Hilbert space $\H$,
  $A:D(A)\subset\H\to\H$ is an unbounded skew-selfadjoint operator on $\H$
  and $F$ is a given source term.   
  
  We are interested in a unique solution $U$ of above equation.
  For this purpose let $\rho> 0$ and define the weighted $L^2$-function space 
  \[
    H_\rho(\R;\H):=
      \left\{ 
          f:\R\to \H\,:\, f \mbox{ meas.}, 
          \int_\R \norm{f(t)}{\H}^2 \exp(-2\rho t) \dt<\infty
      \right\}.
  \]
  The space $H_\rho(\R;\H)$ is a Hilbert space endowed with the natural inner 
  product given by 
  \[
    \scpr{f,g}\,:=
    \int_{\R} \scp{f(t),g(t)} \exp(-2\rho t) \dt
  \]
  for all $f,g\in H_\rho(\R;\H)$, where $\scp{f(t),g(t)}$ is the inner product of $\H$ and $\norm{\cdot}{\H}$ 
  its associated norm. We obtain a norm by setting $\norm{f}{\rho}^2:=\scpr{f,f}$.
  The associated weighted $H^k$-function spaces are denoted by $H_\rho^k(\R;\H)$ for $k\in \N$.
  Now from \cite[Thm. (solution theory)]{Picard} it follows: 
  If there exists a $\rho_0>0$ and a $\gamma>0$ such that for all $\rho\geq\rho_0$ and $x\in \H$
  \begin{gather}\label{eq:pos}
    \scp{(\rho M_0+M_1)x,x}\geq \gamma\scp{x,x}=\gamma\norm{x}{\H}^2,
  \end{gather}
  then for all right hand sides $F\in H_{\rho}(\R,\H)$ exists a unique solution $U\in H_{\rho}(\R,\H)$.
  Furthermore, by above condition $\scp{M_0x,x}\geq 0$ follows and there exists a root $M_0^{1/2}$ of $M_0$.
  Note that the theory presented in \cite{Picard} deals with vanishing initial conditions at $t\to-\infty$.

  \begin{corollary}\label{cor:exist}
    Under the conditions \eqref{eq:pos} and
    \begin{subequations}\label{eq:comp}
      \begin{align}
        &F|_{\R_{\geq 0}}\text{ is continuous and }
        F(t)=0,\,t<0,\\
        &U_0\in\text{dom}(A)
      \intertext{and}
        &(M_1+A)U_0=F(0^+)
      \end{align}
    \end{subequations}
    problem \eqref{eq:evo} has a unique solution $U$ with
    \[
      U(0^+)=U_0.
    \]
  \end{corollary}
  \begin{proof}
    Problem \eqref{eq:evo} given as a problem on $\R$ reads
    \[
       (\partial_tM_0+M_1+A)U=F+\delta_0M_0U_0
    \]
    where the initial condition $M_0U(0^+)=M_0U_0$ is included via the delta distribution $\delta_0$ 
    at $t=0$ on the right-hand side.
    
    Let $H_0$ denote the Heaviside function with the jump at $t=0$. 
    We obtain for $U-H_0U_0$ the evolutionary problem
    \begin{align*}
      (\partial_tM_0+M_1+A)(U-H_0U_0)
        &= F-(M_1+A) H_0 U_0=:\tilde F.
    \end{align*}
    By~\eqref{eq:comp} we have $\tilde F(t)=0,\,t<0$, $\tilde F(0)=0$ and $\tilde F$ is continuous.
    Now \cite{TW19} yields that the problem for $U-H_0U_0$ has a unique
    solution in $H_\rho^1(\R,\H)$. Thus $U$ is a unique solution of \eqref{eq:evo} and $U(0^+)=U_0$.
  \end{proof}
  
  In the following we assume conditions \eqref{eq:pos} and \eqref{eq:comp} to be fulfilled. Then
  $U_0$ is an initial data on the whole $\Omega$, explicitly also in the elliptic and parabolic regime. But due to
  the compatibility condition \eqref{eq:comp} it cannot be chosen independently of $F$.

  In \cite{FrTW16} the class of changing type problems problems was investigated numerically using
  a discontinuous Galerkin approach for the discretisation in time. Here we want to apply 
  a continuous approach, namely the continuous Galerkin-Petrov method \cite{Winther81,AM89,Schieweck10,AM15,AM04,AMN11}.
  
  Note that, like in \cite{FrTW16}, we deal in this paper with problems that have a changing type 
  over the given domain and could be rewritten into second order form as shown above. But then transmission
  conditions would need to be stated that are embedded automatically into the first order formulation.
  This is a very useful feature of the general approach and it allows to combine models from 
  different parts of physics into one well-posed problem. We want to emphasise that the time discretisation
  presented and analysed in this paper holds for all problems of above general class of first order problems, 
  only the spatial discretisation has to be adapted to the operator $A$.
  
  For our problem and operator $A$ the Hilbert space $\H$ and $D(A)$ can now be specified to 
  \[
    \H=L^2(\Omega)\otimes(L^2(\Omega))^n
    \quad\text{and}\quad
    D(A)=H_0^1(\Omega)\otimes H_{\dive}(\Omega).
  \]
  
  \begin{remark}
    The solution theory demands $A$ to be skew-selfadjoint which in turn restricts 
    the choice of boundary data. Some simple choices are homogeneous Dirichlet boundary conditions
    on the first component, encoded by $\grad^\circ$ in above operator $A$, homogeneous Neumann boundary
    conditions on the second component, encoded by $\dive^\circ$ or periodic boundary 
    conditions on both components, encoded by $\grad^\#$ and $\dive^\#$.\\
    Inhomogeneous conditions can always be transformed into homogeneous ones by a 
    substitution changing the right hand side of the problem.
  \end{remark}
        
  The paper is organised as follows.
  The precise formulation of the method considered is stated in Section~\ref{sec:method}
  while Section~\ref{sec:exist} deals with the existence of discrete solutions.
  In Section~\ref{sec:error} we present error estimates and finally
  Section~\ref{sec:numerics} gives some numerical examples and conclusions.
  
\section{Numerical method}\label{sec:method}
  The discrete variational form of \eqref{eq:evo} uses a decomposition of $[0,T]$ into $M$
  disjoint intervals $I_m=(t_{m-1},t_m]$ of length $\tau_m=t_m-t_{m-1}$ for $m\in\{1,\dots,M\}$.
  Furthermore let $\Omega$ be discretised into $\Omega_h$ by a regular simplicial mesh that resolves the sets $\Omega_{\mathrm{ell}},\,
  \Omega_{\mathrm{par}}$ and $\Omega_{\mathrm{hyp}}$, i.e. each of these subdomains is a union of mesh cells,
  and let $h$ be the maximal diameter of the cells of $\Omega_h$. Furthermore, let $r,\,k\geq 1$ denote polynomial degrees.
  
  Then the piecewise polynomial function spaces for the trial and test functions resp. are given by
  \begin{align*}
    \U^\tau_h&:=\{u\in H^1_\rho([0,T],\H):u\big|_{I_m}\in\PS_{r}(I_m,V_1\otimes V_2),m\in\{1,\dots,M\}\},\\
    \V^\tau_h&:=\{v\in H_\rho([0,T],\H):  v\big|_{I_m}\in\PS_{r-1}(I_m,V_1\otimes V_2),m\in\{1,\dots,M\}\},
  \end{align*}
  where the spatial spaces are
  \begin{align*}
    V_1&:=\left\{v\in H_0^1(\Omega):\,v|_\sigma\in\PS_k(\sigma)\,\forall \sigma\in\Omega_h\right\},\\
    V_2&:=\left\{w\in H_{\dive}(\Omega):\,w|_\sigma\in RT_{k-1}(\sigma)\,\forall \sigma\in\Omega_h\right\}
  \end{align*}
  and therefore
  \[
    V_1\otimes V_2\subset D(A)\subset\H.
  \]
  Here $\PS_k(\sigma)$ is the space of polynomials of degree up to $k$ on 
  the cell $\sigma$ of $\Omega_h$ and $RT_{k-1}(\sigma)$ is the Raviart-Thomas-space, defined by
  \[
    RT_{k-1}(\sigma)=(\PS_{k-1}(\sigma))^n+\vx\PS_{k-1}(\sigma)\subset\PS_k(\sigma)^n.
  \]
  Note that we retain the regularity in space of the trial functions also for the test functions 
  in order to define a Galerkin method in space.
  Furthermore, if the mesh consists of quadrilateral or hexahedral cells, in above definitions 
  and statements the polynomial space $\PS_k(\sigma)$ can be replaced by a mapped $\QS_k$-space,
  including all polynomials of total degree $k$ over a reference element mapped onto $\sigma$. 
  If the mesh is a combination of both types of cells, a combination of spaces also works with
  a suitable mapping ensuring the continuities.
  
  Let us localise in addition the scalar product in $H_\rho(\R,\H)$
  to the time intervals $I_m$ by
  \[
    \scprm{f,g}\,:=
    \int_{I_m} \scp{f(t),g(t)} \exp(-2\rho t) \dt
  \]  
  and the norm $\norm{f}{\rho,m}^2:=\scprm{f,f}$.
  Then the variational formulation using the continuous Galerkin-Petrov method reads:\\
  Find $U^\tau_h\in\U^\tau_h$ such that for all $V^\tau_h\in\V^\tau_h$ and $m\in\{1,\dots,M\}$
  \begin{subequations}\label{eq:evodisc}
    \begin{align}
      B_m(U^\tau_h,V^\tau_h)
        &:=\scprm{(\partial_t M_0+M_1+A)U^\tau_h,V^\tau_h}=\scprm{F,V^\tau_h},
    \intertext{where}
      U^\tau_h(0)&=\I U_0\label{eq:evodisc2}
    \end{align}
  \end{subequations}
  is the initial value. Here $\I=(\I_1,\I_2)$ denotes the spatial interpolation operator,
  where $\I_1:H_\rho([0,T],H^1(\Omega))\to H_\rho([0,T],V_1)$ is locally the Scott--Zhang interpolant on each cell $\sigma$, 
  see \cite{SZ90} for a precise definition, and
  $\I_2:H_\rho((0,t),H(\dive,\Omega)\cap(L^s(\Omega))^n)\to H_\rho([0,T],V_2)$ with $s>2$ is the standard interpolator 
  defined via moments, see \cite{BF91}. Note that it is appropriate to include the full initial conditions
  into the discrete problem, see Corollary \ref{cor:exist}.
 
\section{Existence of discrete solution}\label{sec:exist}
  Let us start by defining $\Pi_h^\tau$ as the orthogonal $L^2$-projection w.r.t. $\scpr{\cdot,\cdot}$ into the test space $\V^\tau_h$, i.e.
  \begin{gather}\label{eq:Pi_r}
    \scpr{U-\Pi_h^\tau U,W^\tau_h}=0,\quad\text{for all }W^\tau_h\in\V^\tau_h,
  \end{gather}
  $R$ and $N$ as the projectors onto the range and nullspace of $M_0$, resp, 
  and 
  \[
    \tnorm{U^\tau_h}_\rho^2
      :=\frac{1}{2}\norm{M_0^{1/2} U^\tau_h(T)}{\H}^2\e^{-2\rho T}+
        \norm{N U^\tau_h(0)}{\H}^2+
        \gamma\norm{\Pi_h^\tau U^\tau_h}{\rho}^2.
  \]
  \begin{lemma}
    The seminorm $\tnorm{U^\tau_h}_\rho$ is a norm on $\U^\tau_h$.
  \end{lemma}
  \begin{proof}
    With $\U^\tau_h$ being finite we only have to show that $\tnorm{U^\tau_h}_\rho=0$ implies
    $U^\tau_h=0$. Thus, let us assume $\tnorm{U^\tau_h}_\rho=0$. Then it follows immediately
    $\Pi_h^\tau U^\tau_h=0$ and due to continuity, one degree of freedom is left for $U^\tau_h$.
    On each time interval $U^\tau_h$ is a multiple of a weighted Legendre polynomial
    that is orthogonal to $V^\tau_h$ w.r.t. $\scpr{\cdot,\cdot}$. 
    From $\norm{N U^\tau_h(0)}{\H}=0$ we conclude 
    \[
      N U^\tau_h(0)=0
    \]
    and therefore $N U^\tau_h=0$, because the Legendre polynomial is not zero at the left boundary.
    From $\norm{M_0^{1/2} U^\tau_h(T)}{\H}=0$ we have similarly
    \[
      R U^\tau_h(T)=0
    \]
    and therefore $R U^\tau_h=0$, because the Legendre polynomial is not zero at the right boundary.
    With
    \[
      U_h^\tau=RU_h^\tau+NU_h^\tau=0
    \]
    we have the assertion.
  \end{proof}
  \begin{lemma}\label{lem:exist}
    It holds 
    \[
      \frac{1}{2}\norm{M_0^{1/2} U^\tau_h(T)}{\H}^2\e^{-2\rho T}+
      \gamma\norm{\Pi_h^\tau U^\tau_h}{\rho}^2
      \leq
      \sum_{m=1}^M B_m(U^\tau_h,\Pi_h^\tau U^\tau_h)+\frac{1}{2}\norm{M_0^{1/2} \I U_0}{\H}^2.
    \]
  \end{lemma}
  \begin{proof}
    Let us consider any interval $I_m$. Then it holds
    \[
      B_m(U^\tau_h,\Pi_h^\tau U^\tau_h)
        = \scprm{\partial_t M_0 U^\tau_h,U^\tau_h}+\scprm{M_1 \Pi_h^\tau U^\tau_h,\Pi_h^\tau U^\tau_h},
    \]
    where the skew-symmetry of $A$ and the definition of $\Pi_h^\tau$ was used.
    For the first term we apply integration by parts and obtain due to the exponential weight
    \[
      \scprm{\partial_t M_0 U^\tau_h,U^\tau_h}
        =\rho\scprm{M_0 U^\tau_h,U^\tau_h}+\frac{1}{2}\norm{M_0^{1/2} U^\tau_h(t)}{\H}^2\e^{-2\rho t}\big|_{t_{m-1}}^{t_m}.
    \]
    By the $L^2$-orthogonality \eqref{eq:Pi_r} it follows
    \begin{align*}
      \scprm{M_0 U^\tau_h,U^\tau_h}
        &= \scprm{M_0 (U^\tau_h-\Pi_h^\tau U^\tau_h),U^\tau_h\!-\Pi_h^\tau U^\tau_h}+\scprm{M_0 \Pi_h^\tau U^\tau_h,\Pi_h^\tau U^\tau_h}\\
        &\geq \scprm{M_0 \Pi_h^\tau U^\tau_h,\Pi_h^\tau U^\tau_h}
    \end{align*}
    and therefore
    \[
      \scprm{\partial_t M_0 U^\tau_h,U^\tau_h}
        \geq \scprm{\rho M_0 \Pi_h^\tau U^\tau_h,\Pi_h^\tau U^\tau_h}+\frac{1}{2}\norm{M_0^{1/2} U^\tau_h(t)}{\H}^2\e^{-2\rho t}\big|_{t_{m-1}}^{t_m}.
    \]
    With the general existence assumption $\rho M_0+M_1\geq \gamma$ 
    and $M_0\geq 0$ we obtain
    \begin{gather}
      B_m(U^\tau_h,\Pi_h^\tau U^\tau_h)
        \geq \gamma\norm{\Pi_h^\tau U^\tau_h}{\rho,m}^2+
              \frac{1}{2}\norm{M_0^{1/2} U^\tau_h(t)}{\H}^2\e^{-2\rho t}\big|_{t_{m-1}}^{t_{m}}.\label{eq:lowerbound}
    \end{gather}
    Summing over the intervals the statement follows.
  \end{proof}
  
  It follows
  \begin{align*}
    \tnorm{U^\tau_h}_\rho^2
      &\leq \sum_{m=1}^M B_m(U^\tau_h,\Pi_h^\tau U^\tau_h)+\frac{1}{2}\norm{M_0^{1/2} \I U_0}{\H}^2+\norm{N U^\tau_h(0)}{\H}^2\\
      &=\sum_{m=1}^M \scprm{f,\Pi_h^\tau U^\tau_h}+\frac{1}{2}\norm{M_0^{1/2} \I U_0}{\H}^2+\norm{N U^\tau_h(0)}{\H}^2\\
      &\leq \frac{1}{2\gamma}\norm{f}{\rho}^2+\frac{1}{2}\tnorm{U^\tau_h}_\rho^2+\frac{1}{2}\norm{M_0^{1/2} \I U_0}{\H}^2+\frac{1}{2}\norm{N \I U_0}{\H}^2
  \end{align*}
  and therefore
  \begin{gather}\label{eq:stab}
    \tnorm{U^\tau_h}_\rho^2
      \leq \frac{1}{\gamma}\norm{f}{\rho}^2+\norm{M_0^{1/2} \I U_0}{\H}^2+\norm{N \I U_0}{\H}^2.
  \end{gather}
  This shows unique existence and continuous dependence on $f$ and $U_0$ of the discrete solution $U^\tau_h$.

\section{Error-estimation}\label{sec:error}
  Let us start by stating interpolation error estimates.
  
  \subsection*{Interpolation in time}
    Let $P_r:H_\rho^1([0,T],\H)\to H^1_\rho([0,T],\H)$, where $P_r u\big|_{I_m}\in\PS_{r}(I_m,\H)$ for all $m\in\{1,\dots,M\}$,
    be the interpolation operator fulfilling locally 
    for all $m$ and $v\in H_\rho^1([0,T],\H)$
    \begin{align*}
        (P_r v-v)(t_{m-1})&=0,\quad
        (P_r v-v)(t_{m})   =0,\\
        \scprm{P_r v-v,w} &=0\quad\forall w\in\PS_{r-2}(I_m,\H).
    \end{align*}
    Although we have weighted norms and scalar products the standard interpolation error estimates
    \[
      \norm{P_r v-v}{\rho} \leq C \tau^{r+1}\norm{\partial_t^{r+1} v}{\rho}
    \]
    holds for $v\in H^{r+1}_\rho([0,T],\H)$, where here and further on $C>0$ denotes a generic constant
    and $\tau:=\max\{\tau_m\}$.

  \subsection*{Interpolation in space}
    As previously stated we use $\I=(\I_1,\I_2)$ as spatial interpolation operator,
    where the first component $\I_1:H_\rho([0,T],H^1(\Omega))\to H_\rho([0,T],V_1)$ is the Scott--Zhang interpolant, 
    and the second component $\I_2:H_\rho((0,t),H(\dive,\Omega)\cap(L^\sigma(\Omega))^n)\to H_\rho([0,T],V_2)$ 
    with $\sigma>2$ is the standard Raviart-Thomas interpolator. 
    Here it holds for all $v\in H_0^1(\Omega)\cap H^s(\Omega)$, see \cite{SZ90},
    \begin{equation}\label{eq:inter_H1}
      \norm{v-\I_1 v}{0}        \leq C h^s \norm{v}{r},\qquad
      \norm{\grad(v-\I_1 v)}{0} \leq C h^{s-1} \norm{v}{s},
    \end{equation}
    where $1\leq s \leq k+1$, $\norm{v}{s}$ denotes the $H^s(\Omega)$-norm, and for all $q\in H^s(\Omega)$ 
    such that $\dive q\in H^s(\Omega)$, see \cite{BF91}
    \begin{equation}\label{eq:inter_HDiv}
      \norm{q-\I_2 q}{0}        \leq C h^s \norm{q}{s},\qquad
      \norm{\dive(q-\I_2 q)}{0} \leq C h^s \norm{\dive q}{s},
    \end{equation}
    where $1\leq s\leq k$.
  
  \subsection*{Error analysis}
    Note that we do have for all $V^\tau_h\in\V^\tau_h$ the Galerkin orthogonality
    \begin{gather}\label{eq:Galorth}
      B_m(U-U^\tau_h,V^\tau_h)=0
    \end{gather}
    for the solution $U\in H_\rho^1([0,T],\H)$ of \eqref{eq:evo} 
    and $U^\tau_h\in \U^\tau_h$ of \eqref{eq:evodisc}. We now want to estimate the error $U-U^\tau_h$  
    and decompose it into $U-U^\tau_h=\eta+\xi$, where
    \[
      \eta=\eta_1+\eta_2,\quad
      \eta_1=U-P_rU,\quad
      \eta_2=P_r(U-\I U),\quad
      \xi=P_r\I U-U^\tau_h\in\U^\tau_h.
    \]
    Note that with \eqref{eq:evodisc2} it follows
    \[
      \xi(0)=P_r\I U(0)-U^\tau_h(0)=\I U(0)-\I U(0)=0.
    \]
    \begin{lemma}
      It holds for any $m\in\{1,\dots,M\}$ and $V^\tau_h\in\V^\tau_h$
      \begin{gather}\label{eq:errineq}
        \scprm{(\partial_t M_0+M_1+A)\xi,V^\tau_h}
          \leq\left(\norm{(2\rho M_0+M_1)\eta}{\rho,m}+
                    \norm{A\eta}{\rho,m}\right)\norm{V^\tau_h}{\rho,m}.
      \end{gather}      
    \end{lemma}
    \begin{proof}
      Using the Galerkin orthogonality \eqref{eq:Galorth} we obtain 
      the error equality
      \[
        \scprm{(\partial_t M_0+M_1+A)\xi,V^\tau_h}
          =-\scprm{(\partial_t M_0+M_1+A)\eta,V^\tau_h}.
      \]
      Using integration by parts and the properties of $P_r$ we obtain for all $w\in\V^\tau_h$ and $v\in H_\rho^1([0,T],\H)$
      \begin{align*}
        \scprm{\partial_tM_0(v-P_r v),w}
           = 2\rho\scprm{M_0(v-P_r v),w}&-\scprm{v-P_r v,\partial_t M_0w}\\&+\scp{v-P_r v,w}\e^{-2\rho t}\big|_{t_{m-1}}^{t_m}\\
           = 2\rho\scprm{M_0(v-P_r v),w}&.
      \end{align*}
      Thus we get the error equation
      \begin{gather}\label{eq:erroreq}
        \scprm{(\partial_t M_0+M_1+A)\xi,V^\tau_h}
          =-\scprm{(2\rho M_0+M_1+A)\eta,V^\tau_h}
      \end{gather}
      from which \eqref{eq:errineq} follows by a Cauchy-Schwarz inequality.      
    \end{proof}
    
    From the error equation \eqref{eq:erroreq} and the stability estimate \eqref{eq:stab} we obtain
    \begin{gather}\label{eq:discrerror}
        \gamma\norm{\Pi_h^\tau \xi}{\rho}^2+
      \frac{1}{2}\norm{M_0^{1/2}\xi(T)}{\H}^2\e^{-2\rho T}
        \leq\frac{1}{\gamma}(\norm{(2\rho M_0+M_1)\eta}{\rho}^2+
                              \norm{A\eta}{\rho}^2)
    \end{gather}
    by substituting $U^\tau_h:=\xi$ and $f:=-(2\rho M_0+M_1+A)\eta$, and noting $\xi(0)=0$.
    
    In order to simplify the representation of the main result, let us abbreviate
    \[
      \H^k:=H^k(\Omega)\otimes (H^k(\Omega))^n
    \]
    and
    \[
      \norm{U}{\H^k,\rho}^2
      :=\int_0^T\norm{U(t)}{k}^2\exp(-2\rho t)\dt.
    \]

    \begin{theorem}\label{thm:cts}
      We assume for the solution $U$ of \eqref{eq:evo} the 
      regularity     
      \[
        U\in H_\rho^{1}([0,T];\H^k)\cap H_\rho^{r+1}([0,T];\H) 
      \]
      as well as 
      \[
        AU\in H_\rho([0,T];   \H^k)\cap H_\rho^{r+1}([0,T];\H).
      \]
      Then we have for the error of the numerical solution $U^\tau_h$ of 
      \eqref{eq:evodisc}
      \begin{align*}
        \tnorm{U-U^\tau_h}_\rho
           \leq C\Big(\quad  &\tau^{r+1}\left( \norm{\partial_t^{r+1}U}{\rho}+\norm{\partial_t^{r+1}AU}{\rho} \right)\\ 
                       +&h^k\left( \norm{U}{\H^k,\rho}+\norm{AU}{\H^k,\rho}+\norm{U(T)}{\H^k}\e^{-\rho T}+\norm{NU_0}{\H^k} \right)\Big).
      \end{align*}
    \end{theorem}
    \begin{proof}
      By the decomposition of the norm and the error we have to estimate
      \begin{align*}
        \norm{\Pi_h^\tau (U-U^\tau_h)}{\rho}
          &\leq \norm{\Pi_h^\tau \eta_1}{\rho}+\norm{\Pi_h^\tau \eta_2}{\rho}+\norm{\Pi_h^\tau \xi}{\rho},\\
        \norm{M_0^{1/2}(U-U^\tau_h)(T)}{H}
          &\leq \norm{M_0^{1/2}\eta_1(T)}{H}+\norm{M_0^{1/2}\eta_2(T)}{H}+\norm{M_0^{1/2}\xi(T)}{H}\\
          &=\norm{M_0^{1/2}\eta_2(T)}{H}+\norm{M_0^{1/2}\xi(T)}{H},\\
        \norm{N(U-U^\tau_h)(0)}{H}
          &= \norm{N(U-\I U)(0)}{H}.
      \end{align*}
      Using above interpolation error estimates we obtain
      \begin{align*}
        \norm{\Pi_h^\tau \eta_1}{\rho}
          &=\norm{\eta_1}{\rho}
          \leq C\tau^{r+1}\norm{\partial_t^{r+1} U}{\rho},\\
        \norm{\Pi_h^\tau \eta_2}{\rho}
          &=\norm{\eta_2}{\rho}
          \leq C \norm{U-\I U}{\rho}
          \leq C h^{k}\norm{U}{\H^k,\rho},\\
        \norm{M_0^{1/2}\eta_2(T)}{H}
          &\leq C h^k \norm{U(T)}{\H^k},\\
        \norm{N(U-\I U)(0)}{H}
          &\leq C h^k \norm{NU_0}{\H^k}.
      \end{align*}
      For the remaining two terms we apply \eqref{eq:discrerror} and obtain
      \begin{align*}
        \norm{\Pi_h^\tau \xi}{\rho}
          &\leq C\left( \tau^{r+1}\norm{\partial_t^{r+1} U}{\rho}+h^{k}\norm{U}{\H^k,\rho}+
                        \tau^{r+1}\norm{\partial_t^{r+1} AU}{\rho}+h^{k}\norm{AU}{\H^k,\rho}\right)
      \end{align*}
      and similarly for $\norm{M_0^{1/2}\xi(T)}{H}$. Combining these results proves the error estimate.
    \end{proof}
  \begin{remark}
    We assumed in Theorem~\ref{thm:cts} slightly higher regularity assumptions on $U$ than actually needed. 
    Instead of assuming $U\in H^1_{\rho}([0,T],\H^k)$ for the point evaluation at $t=T$ 
    the weaker assumption $U\in W^{0,\infty}_\rho([0,T],\H^k)$ suffices.
    But in order to prove that claim from conditions on the right-hand side the easiest 
    way is by proving above regularity and using the Sobolev-embedding.
  \end{remark}
  \begin{remark}
    In this section we presented an error analysis for the fully discrete problem 
    of the changing type system. At the same time it is true for all operators $M_0$ and $M_1$
    fulfilling Assumption \eqref{eq:pos}.
    The analysis can also easily be adapted to general evolutionary problems having a different spatial operator $A$
    by defining suitable discrete spatial function spaces and corresponding interpolation operators, 
    and providing sufficient interpolation error estimates.
%
  \end{remark}
  \begin{theorem}\label{thm:cts2}
    In the case of $M_0>0$, e.g. a purely hyperbolic problem, we can also give a convergence result
    in the weighted $L^2$-type $\norm{\cdot}{\rho}$-norm. Under the same conditions as in Theorem~\ref{thm:cts}
    we have
    \begin{align*}
      \norm{U-U^\tau_h}{\rho}
          \leq C\sqrt{1+T}\big[\quad  &\tau^{r+1}\left( \norm{\partial_t^{r+1}U}{\rho}+\norm{\partial_t^{r+1}AU}{\rho} \right)\\ 
                            +&h^k\left( \norm{U}{\H^k,\rho}+\norm{AU}{\H^k,\rho}+\norm{\partial_t U}{\H^k,\rho}+\norm{NU_0}{\H^k} \right)\big].
    \end{align*}    
  \end{theorem}
  \begin{proof}    
    For this result we need
    \begin{itemize}
      \item a local norm equivalence for all $W_h^\tau\in\U_h^\tau$
            \[
              \norm{W_h^\tau}{\rho,m}^2
                \leq C_1\left( \gamma\norm{\Pi_h^\tau W_h^\tau}{\rho,m}^2+\tau_m\norm{M_0^{1/2}W_h^\tau(t_m)}{H}^2\e^{-2\rho t_m} \right)
            \]
            with a constant $C_1$ independent of $\tau_m$ and $W^\tau_h$, that holds true because 
            $\Pi_h^\tau W^\tau_h-W^\tau_h$ is a multiple of a weighted Legendre polynomial 
            of degree $r$, $t_m$ is not a zero of it and the scaling w.r.t. $\tau_m$ of the two terms is the same,
      \item a local estimation of the discrete error $\xi$ with a localisation of the norms to the interval $[0,t_m]$ instead of $[0,T]$
            \begin{align*}
              &\norm{\xi}{\rho,[0,t_m]}^2
               +\frac{1}{2}\norm{M_0^{1/2}\xi(t_m)}{H}^2\e^{-2\rho t_m}\\
              &\leq C\Big(\quad\tau^{2(r+1)}\left( \norm{\partial_t^{r+1}U}{\rho,[0,t_m]}^2+\norm{\partial_t^{r+1}AU}{\rho,[0,t_m]}^2 \right)\\ 
              &\hspace*{0.79cm}+h^{2k}\!\left(\! \norm{U}{\H^k,\rho,[0,t_m]}^2\!+\!\norm{AU}{\H^k,\rho,[0,t_m]}^2\!+\!\norm{U(t_m)}{\H^k}^2\e^{-2\rho T}\!+\!\norm{NU_0}{\H^k}^2\! \right)\!\!\Big),
            \end{align*}
            that follows by the same lines as in Thm.~\ref{thm:cts},
      \item a Sobolev embedding for $t_n<t_m$ and $U\in H^1_\rho([t_n,t_m],\H)$
            \[
              \norm{U(t_m)}{H}^2\e^{-2\rho t_m}
                \leq C_{inv} \left(\frac{1}{t_m-t_n} \norm{U}{\rho,[t_n,t_m]}^2+(t_m-t_n)\norm{\partial_t U}{\rho,[t_n,t_m]}^2 \right)
            \]
            with a constant $C_{inv}$ independent of $U$, $t_n$ and $t_m$.
    \end{itemize}
    Then it follows
    \begin{align*}
      \norm{\xi}{\rho}^2
        &\leq C_1\left( \gamma\norm{\Pi_h^\tau \xi}{\rho}^2+\sum_{m=1}^M\tau_m\norm{M_0^{1/2}\xi(t_m)}{H}^2\e^{-2\rho t_m} \right)\\
        &\leq C_1(1+T)\bigg[\tau^{2(r+1)}\left( \norm{\partial_t^{r+1}U}{\rho}^2+\norm{\partial_t^{r+1}AU}{\rho}^2 \right)\\
        &\hspace*{2.2cm}+h^{2k}\bigg( \left( 1+\frac{C_{inv}}{T}\right) \norm{U}{\H^k,\rho}^2+\norm{AU}{\H^k,\rho}^2\\
        &\hspace*{3.5cm}              +C_{inv}\norm{\partial_t U}{\H^k,\rho}^2
                                      +\norm{NU_0}{\H^k}^2
                               \bigg)
                      \bigg],
    \end{align*}
    where the Sobolev embedding for $\norm{U(T)}{\H^k}\e^{-\rho T}$ uses the whole interval $[0,T]$ and 
    only $[t_{m-1},t_m]$ for $\norm{U(t_m)}{\H^k}\e^{-\rho t_m}$, as well as $\tau_m\leq 1$.
    Together with the interpolation error bound 
    \[
      \norm{\eta}{\rho}
        \leq\norm{\eta_1}{\rho}+\norm{\eta_2}{\rho}
        \leq C\left( \tau^{r+1}\norm{\partial_t^{r+1} U}{\rho}+h^k\norm{U}{\H^k,\rho} \right)
    \]
    the claim follows.
  \end{proof}

\section{Numerical examples}\label{sec:numerics}
  We consider two examples with unknown solutions. Simulations with known smooth solutions 
  were also made and the theoretical orders were observed. The two following examples 
  show a more realistic behaviour in the case of changing type systems. That both examples have initial 
  values zero is not a restriction. We look into the convergence behaviour
  also w.r.t. the weighted $L^2$-norm $\norm{\cdot}{\rho}$ in addition to the $\tnorm{\cdot}_\rho$-norm 
  in order to compare the results with those of the discontinuous Galerkin method from \cite{FrTW16}.
  In the finite discrete setting both norms are equivalent.
  
  All computations were done in the finite-element framework $\mathbb{SOFE}$\footnote{\texttt{github.com/SOFE-Developers/SOFE}}.
  \subsection{1+1d example}
    Let us consider as first example one spatial dimension and combine a hyperbolic and an elliptic region.
    To be more precise, let $\Omega=\left[-\pi,\pi\right]$, $\Omega_{\mathrm{hyp}}= \left[-\pi,0\right]$, 
    and $\Omega_{\mathrm{ell}}=\left[0,\pi\right]$. As final time we set $T=4\pi$. The problem is now given
    by
    \begin{gather}\label{ex:1d}
      \left[
        \partial_t
        \begin{pmatrix}[c]
        \id{\Omega_{\mathrm{hyp}}} & 0\\
        0 & \id{\Omega_{\mathrm{hyp}}}
        \end{pmatrix}
        +\begin{pmatrix}[c]
          \id{\Omega_{\mathrm{ell}}} & 0\\
          0 & \id{\Omega_{\mathrm{ell}}}
        \end{pmatrix}
        +\begin{pmatrix}[c]
          0 & \partial_x\\
          \mathring{\partial_x} & 0
        \end{pmatrix}
      \right]U
        =F
    \end{gather}
    with homogeneous Dirichlet-conditions for the first component of $U:\R\times\R\to\R\times\R$,
    the initial condition $U_0=0$
    and a right-hand side $F(t,x)=(f(t,x),g(t,x))\cdot\id{\geq 0}(t)$, where $\id{\geq 0}(t)$ 
    is the characteristic function of the non-negative time line and
    \begin{align*}
      f(t,x) &= \frac{1}{5}\sin(3t)+\min\{t,\pi\}\cos(3x),\\
      g(t,x) &= \sin(t)\left(1-\frac{x^2}{\pi^2}\right).
    \end{align*}
    Thus, $F$ is continuous on $\R$ and it holds $F(t)=0$ for $t\leq 0$. 
    Therefore, the solution theory of \cite{Picard} gives the existence of a 
    unique solution $U$ that is continuous in time. Figure~\ref{fig:sol1d}
    \begin{figure}
     \begin{center}
        \includegraphics[width=0.35\textwidth]{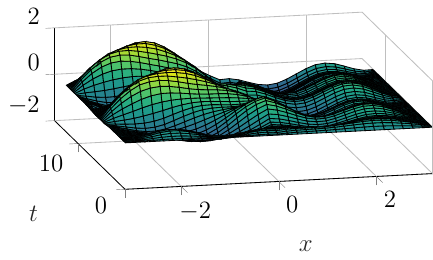}
        \includegraphics[width=0.35\textwidth]{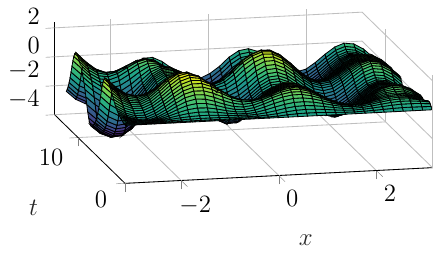}
     \end{center}
      \caption{Solution of problem \eqref{ex:1d}, first component (left) and second component (right)\label{fig:sol1d}}
    \end{figure}
    shows plots of the components of the solution in the domain. Note that 
    the first component has a kink along $x=0$ -- it is continuous but not 
    differentiable in $x$. As mesh we use an equidistant mesh of $N$ cells in $\Omega$
    and $M$ cells in $[0,T]$. In order to calculate the errors we use a reference 
    solution $U_{ref}$ instead of the unknown solution $U$. The reference solution is computed 
    on an $4096\times2048$ mesh with polynomial degrees $k=4$ and $r=3$.
    Table~\ref{tab:ex1d}
    \begin{table}
      \caption{Errors and rates for example \eqref{ex:1d}\label{tab:ex1d}}
      \begin{center}
        \begin{tabular}{rllllll}
          \toprule
          &\multicolumn{4}{c}{cGP-method}
          &\multicolumn{2}{c}{dG-method}\\
          \midrule
          $M=2N$
          &\multicolumn{2}{c}{$\tnorm{U_{ref}-U^\tau_h}_\rho$}
          &\multicolumn{2}{c}{$\norm{U_{ref}-U^\tau_h}{\rho,[0,T]}$}
          &\multicolumn{2}{c}{$\norm{U_{ref}-U^\tau_h}{\rho,[0,T]}$}\\
          \midrule
          \multicolumn{7}{c}{$k=2$, $r=1$}\\
          \midrule
            256 & 2.120e-02 &      & 8.890e-04 &      & 1.808e-04 &     \\
            512 & 5.746e-03 & 1.88 & 3.136e-04 & 1.50 & 7.751e-05 & 1.22\\
           1024 & 1.787e-03 & 1.68 & 1.380e-04 & 1.18 & 3.496e-05 & 1.15\\
           2048 & 7.036e-04 & 1.35 & 6.739e-05 & 1.03 & 1.580e-05 & 1.15\\
          \midrule
          \multicolumn{7}{c}{$k=3$, $r=2$}\\
          \midrule
            256 & 8.806e-04 & 1.05 & 1.187e-04 &      & 6.058e-05 &     \\
            512 & 4.163e-04 & 1.08 & 5.489e-05 & 1.11 & 2.642e-05 & 1.20\\
           1024 & 1.906e-04 & 1.13 & 2.492e-05 & 1.14 & 1.137e-05 & 1.22\\
           2048 & 8.581e-05 & 1.15 & 1.114e-05 & 1.16 & 4.669e-06 & 1.28\\
           \bottomrule
        \end{tabular}        
      \end{center}
    \end{table}
    shows the results for different values of $M$ and $N$ and polynomial degrees $k$ and $r$.
    We coupled $k=r+1$ as the theory gives for smooth $U$ the convergence order $\min\{k,r+1\}$
    if $N$ and $M$ are proportional. We observe for the continuous Galerkin-Petrov method 
    only a convergence rate between 1 and 2 in both norms. Increasing the polynomial
    degree reduces the error, but does not improve the rate much. A reason for this behaviour 
    could be that $U$ is not smooth enough for the error estimates to hold due having 
    jumping coefficients in space and a non-differentiable right hand side. Unfortunately the 
    exact solution to this problem and thus its precise regularity is unknown.\\
    For comparison we also computed approximations with the discontinuous Galerkin method 
    from \cite{FrTW16} that uses globally discontinuous piecewise polynomials of degree $r$ 
    in time and the same approximation in space as the method described in this paper. 
    The errors given in the remaining columns show a similar behaviour 
    with convergence rates between 1 and 2. Nevertheless, the errors are smaller 
    for the discontinuous approach.
    
  \subsection{1+2d example}
    As second example we consider the last example of \cite{FrTW16}.
    Let $T=5.2$, $\Omega=(0,1)^2\subset\R^2$, 
    $\Omega_{\mathrm{hyp}}=\left( \frac{1}{4},\frac{3}{4}\right)^2$ and 
    $\Omega_{\mathrm{ell}}=\Omega\setminus\bar\Omega_{\mathrm{hyp}}$
    The problem is given by
    \begin{gather}\label{ex:2d}
      \left[
      \partial_t
      \begin{pmatrix}
        \id{\Omega_{\mathrm{hyp}}} & 0\\
        0 & \id{\Omega_{\mathrm{hyp}}}
      \end{pmatrix}
      +\begin{pmatrix}
        \id{\Omega_{\mathrm{ell}}} & 0\\
        0 & \id{\Omega_{\mathrm{ell}}}
      \end{pmatrix}
      +\begin{pmatrix}
        0 & \dive \\
        \grad^\circ & 0
      \end{pmatrix}
      \right]U
      =\begin{pmatrix}
        f\\
        0
      \end{pmatrix},
    \end{gather}
    where
    \[
      f(t,\vx)=2\sin(\pi t)\cdot\id{\R_{<1/2}\times\R}(\vx).
    \]
    Figure~\ref{fig:sol2d} 
    \begin{figure}[tb]
      \begin{center}
        \includegraphics[width=0.25\textwidth]{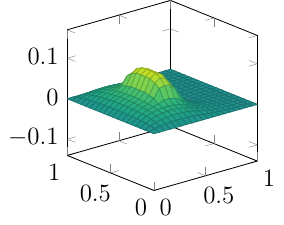}\quad
        \includegraphics[width=0.25\textwidth]{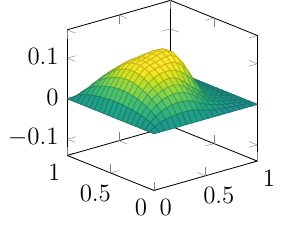}\quad
        \includegraphics[width=0.25\textwidth]{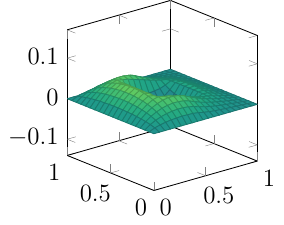}\\
        \includegraphics[width=0.25\textwidth]{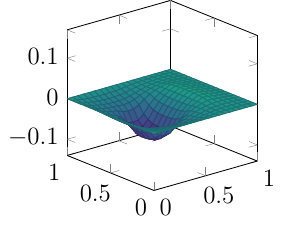}\quad
        \includegraphics[width=0.25\textwidth]{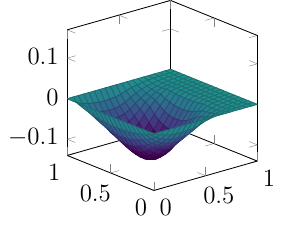}\quad
        \includegraphics[width=0.25\textwidth]{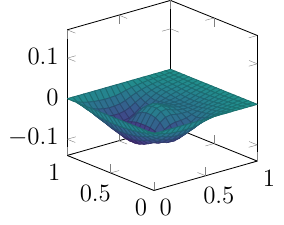}
      \end{center}
      \caption{First component of $U$ at times $t=5\ell/16$ for $\ell\in\{1,\dots, 6\}$ (top left 
               to bottom right) of problem \eqref{ex:2d}}\label{fig:sol2d}
    \end{figure}
    shows some snapshots of the first component of the solution $U:\R\times\R^2\to\R\times\R^2$,
    approximated by a numerical simulation. Again we use equidistant meshes
    with $N$ cells in each dimension of space and $M$ cells in $[0,T]$.
    As reference solution $U_{ref}$ replacing the unknown exact solution we use
    an approximation calculated with $M=192,\,N=96,\,k=3,\,r=2$ and $M=128,\,N=64,\,k=4,\,r=3$, resp.
    Table~\ref{tab:ex2d}
    \begin{table}
      \caption{Errors $\norm{U_{ref}-U^\tau_h}{\rho,[0,T]}$ and rates for example \eqref{ex:2d}\label{tab:ex2d}}
      \begin{center}   
        \begin{tabular}{rllllll}
          \toprule
          &\multicolumn{4}{c}{cGP-method}
          &\multicolumn{2}{c}{dG-method}\\
          \midrule
          $M=2N$
          &\multicolumn{2}{c}{$\tnorm{U_{ref}-U^\tau_h}_\rho$}
          &\multicolumn{2}{c}{$\norm{U_{ref}-U^\tau_h}{\rho,[0,T]}$}
          &\multicolumn{2}{c}{$\norm{U_{ref}-U^\tau_h}{\rho,[0,T]}$}\\
          \midrule
          \multicolumn{7}{c}{$k=2$, $r=1$}\\
          \midrule
          16 & 3.989e-02 &      & 1.961e-02 &      & 7.821e-03 &     \\
          32 & 1.972e-02 & 1.02 & 9.199e-03 & 1.09 & 3.018e-03 & 1.37\\
          64 & 9.435e-03 & 1.06 & 3.751e-03 & 1.29 & 8.813e-04 & 1.78\\
          96 & 5.603e-03 & 1.29 & 1.324e-03 & 1.50 & 2.920e-04 & 1.59\\
          \midrule
          \multicolumn{7}{c}{$k=3$, $r=2$}\\
          \midrule
          16 & 1.041e-02 &      & 5.499e-03 &      & 2.790e-03 &     \\
          32 & 3.689e-03 & 1.50 & 1.435e-03 & 1.94 & 6.385e-04 & 2.13\\
          64 & 1.248e-03 & 1.56 & 4.430e-04 & 1.70 & 2.248e-04 & 1.51\\
          \bottomrule
        \end{tabular}
      \end{center}    
    \end{table}
    shows the results. Similarly to the previous example we do not achieve the optimal 
    convergence order for both methods. Here the data and the right-hand side have jumps 
    along interior lines which reduces the maximum regularity of the solution. 
    Again the discontinuous Galerkin method has smaller errors.
    
\section*{Conclusions}
  The continuous solution of an evolutionary system with continuous right hand side
  can be approximated by several methods. Here we investigated the continuous Galerkin-Petrov
  method, that has optimal convergence order for smooth solutions in the $\tnorm{\cdot}_\rho$-norm. 
  The benefit of the continuous method compared to the discontinuous Galerkin method is the continuity that implies a 
  non-dissipative behaviour. In our examples with unknown solutions, that are probably not smooth enough,
  the discontinuous Galerkin method is slightly better. Furthermore, these examples show that 
  an increase of the polynomial degree in space over 2 and in time over 1 gives no huge benefit. 
  This is different for smooth solutions -- here both methods achieve the theoretical high convergence orders.
  
\section*{Acknowledgement}
  The author would like to thank R. Picard for many helpful discussions on the existence of 
  solution for evolutionary problems and the anonymous reviewer for the helpful comments.
  Also thanks goes to L. Ludwig for providing $\mathbb{SOFE}$.

\bibliographystyle{plain}
\bibliography{lit}

\end{document}